\documentclass[11pt]{article}

\usepackage{amsmath,amssymb,amsthm}
\usepackage[hidelinks]{hyperref}
\usepackage{microtype}
\usepackage{pgfplots}
\pgfplotsset{compat=1.17}

\newtheorem{theorem}{Theorem}[section]
\newtheorem{lemma}[theorem]{Lemma}
\newtheorem{proposition}[theorem]{Proposition}
\newtheorem{corollary}[theorem]{Corollary}
\theoremstyle{definition}
\newtheorem{definition}[theorem]{Definition}
\newtheorem{example}[theorem]{Example}
\theoremstyle{remark}
\newtheorem{remark}[theorem]{Remark}

\DeclareMathOperator{\Lip}{Lip}
\DeclareMathOperator{\sgn}{sgn}

\title{Logically Contractive Mappings: Fixed Points and Event-Indexed Rates}
\author{Faruk Alpay\thanks{Department of Mathematics, Lightcap Institute. Email: \texttt{alpay@lightcap.ai}} \and Taylan Alpay\thanks{Department of Aerospace Engineering, Turkish Aeronautical Association}}
\date{}

\begin{document}
\maketitle

\begin{abstract}
We introduce \emph{logically contractive mappings}---nonexpansive self-maps that contract along a subsequence of iterates---and prove a fixed-point theorem that extends Banach's principle. We obtain event-indexed convergence rates and, under bounded gaps between events, explicit iteration-count rates. A worked example shows a nonexpansive map whose square is a strict contraction, and we clarify relations to Meir--Keeler and asymptotically nonexpansive mappings. We further generalize to variable-factor events and show that $\prod_k \lambda_k = 0$ (equivalently $\sum_k -\ln \lambda_k = \infty$) implies convergence. These results unify several generalized contraction phenomena and suggest new rate questions tied to event sparsity.
\end{abstract}

\section{Introduction}
Banach's contraction mapping theorem is a cornerstone of metric fixed-point theory: any map $T$ on a complete metric space $(X,d)$ with a uniform Lipschitz constant $L<1$ has a unique fixed point $z\in X$, and moreover $T^n x$ converges to $z$ for every $x\in X$ with an exponential rate $L^n d(x,z)$. Numerous extensions have been studied, weakening the strict contraction requirement while retaining fixed-point existence and convergence guarantees. Classic examples include \emph{Meir--Keeler contractions} \cite{2} (which relax uniform $L<1$ to a certain $\varepsilon$-$\delta$ decay condition) and \emph{asymptotically nonexpansive} maps (which allow the Lipschitz constant of $T^n$ to approach $1$ as $n\to\infty$) \cite{3}.

In this note, we formalize a simple unifying concept: a mapping that is overall nonexpansive ($\Lip(T)\le 1$) but occasionally acts like a contraction. Specifically, a \textit{logically contractive} mapping is one for which there exists some $0<\lambda<1$ and an infinite sequence of iterate counts $n_1 < n_2 < n_3 < \cdots$ such that each $T^{n_k}$ is $\lambda^k$-Lipschitz. Intuitively, $T$ is not required to contract every time step, but it does so at an infinite sequence of ``event'' iterations, eventually forcing convergence. We prove that any logically contractive map has a unique fixed point and that all iterations converge to it (Theorem~\ref{thm:fixedpoint}). In fact, the distance to the fixed point decays at least geometrically along the event subsequence, and if the gap between consecutive events is bounded, one can derive an explicit per-iteration decay rate (Corollary~\ref{cor:rate}). We also compare logically contractive maps with the Meir--Keeler and asymptotically nonexpansive classes, showing that none of these notions subsume the others (Proposition~\ref{prop:incomparability}). A concrete example is provided (Example~\ref{ex:T-square}) of a nonexpansive mapping that is not a strict contraction yet is logically contractive (in fact, its second iterate is constant), illustrating the concept. Finally, we generalize the logical contractiveness idea to allow varying contraction factors at different events. We show that if the product of these factors tends to zero (equivalently, an infinite sum of their logarithmic decreases diverges), then $T^n x$ still converges to the fixed point for any starting point $x$ (Theorem~\ref{thm:vlc}). We conclude with brief remarks on the scope of our results and open questions.

\paragraph{Notation.} Throughout, $(X,d)$ is a complete metric space. For a map $T:X\to X$ and any $n\in \mathbb{N}$, $T^n$ denotes the $n$-fold composition (so $T^1=T$). We write $\Lip(S)$ for the (global) Lipschitz constant of a function $S$, i.e.\ $\Lip(S) := \sup_{x\neq y} \frac{d(Sx,Sy)}{d(x,y)}$. A map is \emph{nonexpansive} if $\Lip(T)\le 1$. Compositions of nonexpansive maps are nonexpansive as well, i.e.\ $\Lip(T^n)\le 1$ for any $n$ if $\Lip(T)\le 1$. We write $\lfloor t\rfloor$ for the floor of $t$ (the greatest integer $\le t$).

\section{Fixed-point theorem for logically contractive mappings}
\begin{definition}[Logically contractive mapping]\label{def:logical-contractiveness}
Let $T:X\to X$ be a self-map on a metric space $(X,d)$. We say $T$ is \emph{logically contractive} if $T$ is nonexpansive and there exist a constant $\lambda \in (0,1)$ and an increasing sequence of natural numbers $n_1<n_2<n_3<\cdots$ such that for all $k\ge 1$,
\[
\Lip\!\big(T^{\,n_k}\big) \;\le\; \lambda^k.
\]
In other words, $T^{\,n_k}$ is a strict contraction, with Lipschitz constant at most $\lambda^k$, at an infinite sequence of iterate indices $n_k$.
\end{definition}

By definition, a logically contractive map $T$ has at least one ``event'' (the case $k=1$) where a strict contraction occurs. In particular, $\Lip(T^{\,n_1}) \le \lambda < 1$, so $T^{\,n_1}$ is itself a Banach contraction on $X$. Therefore $T^{\,n_1}$ has a unique fixed point in $X$, which we will denote $z$:
\[
T^{\,n_1}(z) = z.
\]

\begin{theorem}\label{thm:fixedpoint}
Let $T:X\to X$ be a logically contractive mapping. Then:
\begin{enumerate}
\item $T$ has a unique fixed point $z\in X$.
\item For every $x\in X$, the sequence $T^n x$ converges to $z$ as $n\to\infty$.
\item Moreover, for every $x\in X$ and all $k\ge1$,
\[
d\!\big(T^{\,n_k}x,\;z\big) \;\le\; \lambda^k\, d(x,z).
\]
\end{enumerate}
\end{theorem}

\begin{proof}
\emph{(1) Existence/uniqueness.}
Since $\Lip(T^{\,n_1})\le \lambda<1$, $T^{\,n_1}$ is a Banach contraction and hence has a unique fixed point $z\in X$. Now $T^{\,n_1}(Tz)=T(T^{\,n_1}z)=Tz$, so by uniqueness of the fixed point of $T^{\,n_1}$ we have $Tz=z$. If $Tx=x$ for some $x$, then $T^{\,n_1}x=x$, whence $x=z$ by uniqueness.

\emph{(2) Convergence of the full sequence.}
By (3) below, $d(T^{\,n_k}x,z)\le \lambda^k d(x,z)\to 0$ as $k\to\infty$. If $n_k\le n < n_{k+1}$, then by nonexpansiveness of $T^{\,n-n_k}$,
\[
d(T^n x,z) \;=\; d\!\big(T^{\,n-n_k}(T^{\,n_k}x),\;T^{\,n-n_k}(z)\big)\;\le\; d(T^{\,n_k}x,z).
\]
Hence $d(T^n x,z)\to 0$ as $n\to\infty$.

\emph{(3) Eventwise estimate.}
Since $T^{\,n_k}z=z$ and $\Lip(T^{\,n_k})\le \lambda^k$,
\[
d(T^{\,n_k}x,z)\;=\; d\!\big(T^{\,n_k}x,\;T^{\,n_k}z\big)\;\le\; \lambda^k\, d(x,z).
\]
\end{proof}

\begin{remark}[Role of completeness]
Completeness of $(X,d)$ is used only to invoke Banach's theorem for the strict contraction $T^{\,n_1}$; all subsequent steps rely on nonexpansiveness and the eventwise Lipschitz bounds.
\end{remark}

From part (3) we obtain a convergence rate in terms of the event count $k$. To express a per-iteration rate, we need information on the spacing of events.

\begin{corollary}[Iteration-count rate under bounded event gaps]\label{cor:rate}
Assume $T$ is logically contractive with constant $\lambda\in(0,1)$ and let $z$ be the fixed point of $T$. If there is a uniform bound $M$ on event gaps (i.e.\ $n_{k+1}-n_k\le M$ for all $k$), then for every $x\in X$ and every $n\ge n_1$,
\[
d\!\big(T^n x,\;z\big) \;\le\; \lambda^{\,1+\left\lfloor \frac{\,n-n_1\,}{M}\right\rfloor} \, d(x,z).
\]
\end{corollary}

\begin{proof}
Let $m:=\left\lfloor \frac{n-n_1}{M}\right\rfloor$. Then $n_1+mM\le n < n_1+(m+1)M$. 
By induction from $n_{k+1}\le n_k+M$ we have $n_{m+1}\le n_1+mM$, hence $n_{m+1}\le n$. 
By Theorem~\ref{thm:fixedpoint}(3) and nonexpansiveness,
\[
d(T^n x,z)\;\le\; d(T^{\,n_{m+1}}x,z)\;\le\; \lambda^{\,m+1} d(x,z).
\]
\end{proof}

We now show that a bounded-gap schedule can always be enforced by using the first strict event repeatedly.

\begin{lemma}[Bounded gaps without loss of generality]\label{lem:bounded-gaps}
Let $T$ be logically contractive with constant $\lambda$ and first event time $n_1$. Then for every integer $m\ge 1$,
\[
\Lip\!\big(T^{\,m n_1}\big)\;=\;\Lip\!\big((T^{\,n_1})^m\big)\;\le\;\big(\Lip(T^{\,n_1})\big)^m \;\le\; \lambda^m.
\]
Thus the \emph{canonical} schedule $\tilde n_m:= m n_1$ certifies logical contractiveness with constant gap $\tilde n_{m+1}-\tilde n_m = n_1$.
\end{lemma}

\begin{proof}
The identity $(T^{\,n_1})^m=T^{\,m n_1}$ is equality of maps. Submultiplicativity under composition gives
\[
\Lip\!\big((T^{\,n_1})^m\big)\;\le\;\big(\Lip(T^{\,n_1})\big)^m,
\]
and equality need not hold in general (indeed, strict inequality is typical).
\end{proof}

\begin{corollary}[Explicit per-iterate rate]\label{cor:rate-nogap}
For any logically contractive $T$ with first event time $n_1$ and fixed point $z$, we have for all $x\in X$ and all $n\ge 0$,
\[
d\!\big(T^n x,\;z\big) \;\le\; \big(\Lip(T^{\,n_1})\big)^{\left\lfloor \frac{n}{\,n_1\,}\right\rfloor} \, d(x,z)
\;\le\; \lambda^{\left\lfloor \frac{n}{\,n_1\,}\right\rfloor} \, d(x,z).
\]
\end{corollary}

\begin{proof}
Apply Corollary~\ref{cor:rate} to the canonical schedule $\tilde n_m= m n_1$ with gap $M=n_1$.
\end{proof}

\begin{remark}
The original schedule $\{n_k\}$ may be irregular; the canonical schedule in Lemma~\ref{lem:bounded-gaps} produces a clean rate. If later events have much smaller Lipschitz constants than $\lambda$, one can restart with a smaller contraction constant to sharpen the bound.
\end{remark}

\begin{remark}[Equivalent formulation]\label{rem:equiv}
A mapping $T$ is logically contractive if and only if it is nonexpansive and there exists at least one iterate $T^N$ that is a strict contraction (i.e.\ $\Lip(T^N)<1$). Indeed, the ``only if'' part is $k=1$ of Definition~\ref{def:logical-contractiveness}. Conversely, if $\Lip(T^N)=\mu<1$, then by Lemma~\ref{lem:bounded-gaps} the schedule $n_k=kN$ satisfies $\Lip(T^{n_k})\le \mu^k$, which is logically contractive with $\lambda=\mu$.
\end{remark}

\begin{example}[Non-\texorpdfstring{$\mathbb{R}$}{R} setting: $\ell_\infty$]\label{ex:linfty}
Let $X=\ell_\infty(\mathbb{N})$ with the sup norm, and define $t:\mathbb{R}\to\mathbb{R}$ by
\[
t(u)=
\begin{cases}
0, & |u|\le 1,\\
u-\sgn(u), & 1<|u|<2,\\
\sgn(u), & |u|\ge 2~.
\end{cases}
\]
Define $T:X\to X$ coordinatewise by $(Tx)_i:=t(x_i)$. Since $t$ is $1$-Lipschitz on $\mathbb{R}$, we have
$\|Tx-Ty\|_\infty=\sup_i |t(x_i)-t(y_i)|\le \sup_i |x_i-y_i|=\|x-y\|_\infty$, so $T$ is nonexpansive. Moreover, $t^2\equiv 0$ implies $T^2\equiv 0$. Hence $T$ is logically contractive with $n_1=2$ and unique fixed point $0\in X$.
\end{example}

\section{Comparisons with other contraction conditions}
A mapping $T$ is a \emph{Meir--Keeler contraction} \cite{2} if for every $\varepsilon>0$ there exists $\delta>0$ such that
\[
d(x,y)\in [\varepsilon,\varepsilon+\delta) \;\Rightarrow\; d(Tx,Ty) < \varepsilon.
\]
Following Goebel--Kirk \cite{3}, $T$ is \emph{asymptotically nonexpansive} if there exist $k_n\ge 1$ with $k_n\to 1$ and
\[
d(T^n x,\,T^n y)\;\le\; k_n\, d(x,y)\qquad(\forall\,x,y\in X,\ \forall\,n\in\mathbb{N}).
\]
(In particular, this implies $\limsup_{n\to\infty}\Lip(T^n)\le 1$; some authors require $\Lip(T^n)\to 1$ as a definition.)

\begin{proposition}\label{prop:incomparability}
Logical contractiveness is incomparable with Meir--Keeler and with asymptotically nonexpansive mappings:
\begin{enumerate}
\item There exists a logically contractive map that is not Meir--Keeler.
\item There exists a Meir--Keeler contraction that is not logically contractive.
\item There exists an asymptotically nonexpansive map that is not logically contractive.
\end{enumerate}
Moreover, under the Goebel--Kirk definition, every nonexpansive map (hence every logically contractive map) is trivially asymptotically nonexpansive by taking $k_n\equiv 1$.
\end{proposition}

\begin{proof}
\textbf{(1) LC $\not\Rightarrow$ MK.} Define $T:\mathbb{R}\to\mathbb{R}$ by
\[
T(x)=
\begin{cases}
0, & |x|\le 1,\\[3pt]
x-\sgn(x), & 1<|x|<2,\\[3pt]
\sgn(x), & |x|\ge 2~.
\end{cases}
\]
\emph{Nonexpansiveness.} Let $-2,-1,1,2$ be the breakpoints. For $x<y$, refine $\{x,y\}$ by inserting any of these that fall in $[x,y]$, producing $x=t_0<t_1<\dots<t_m=y$. On each $[t_{j-1},t_j]$, $T$ is either constant (Lipschitz $0$) or affine with slope $1$, hence $|T(t_j)-T(t_{j-1})|\le t_j-t_{j-1}$. Summing,
\[
|T(y)-T(x)| \le \sum_{j=1}^m |T(t_j)-T(t_{j-1})|
\le \sum_{j=1}^m (t_j-t_{j-1}) = y-x,
\]
so $\Lip(T)\le 1$. By symmetry the same holds for $y<x$.

\emph{Finite-time convergence.} If $|x|\le 1$ then $T(x)=0$; if $1<|x|<2$ then $T(x)=x-\sgn(x)\in(-1,1)$ so $T^2(x)=0$; if $|x|\ge 2$ then $T(x)=\sgn(x)\in\{-1,1\}$ and $T^2(x)=0$. Thus $T^2\equiv 0$.

\emph{Not Meir--Keeler.} Fix any $\varepsilon\in(0,1]$ and set $x=1$, $y=1+\varepsilon$. Then $d(x,y)=\varepsilon$, but $Tx=0$ and $Ty=\varepsilon$, so $d(Tx,Ty)=\varepsilon\not<\varepsilon$, violating the MK condition for every $\delta>0$.

\smallskip
\textbf{(2) MK $\not\Rightarrow$ LC.} On $X=[0,1]$, let $T(x)=x-c\!\left(x-\tfrac12\right)^3$ with any $c\in (0,4/3]$. Then $T$ is increasing and $C^1$ with
\[
T'(x)=1-3c\!\left(x-\tfrac12\right)^2 \in [\,1- \tfrac{3c}{4},\,1\,]\subseteq[0,1],
\]
so $T$ is nonexpansive. We show $T$ is Meir--Keeler with an explicit $\delta(\varepsilon)$. For $0<\varepsilon\le 1$ choose $\delta:=\tfrac{c}{8}\varepsilon^3$. If $x<y$ and $\varepsilon\le y-x<\varepsilon+\delta$, then by the fundamental theorem of calculus,
\[
T(y)-T(x)=\int_x^y \Big(1-3c(t-\tfrac12)^2\Big)\,dt
= (y-x) - 3c\!\int_x^y (t-\tfrac12)^2\,dt.
\]
Among all intervals of length $h:=y-x$, the integral is minimized when the interval is centered at $\tfrac12$, giving $\int_x^y (t-\tfrac12)^2 dt \ge h^3/12$. Thus
\[
T(y)-T(x) \le h - \frac{c}{4}h^3 \le (\varepsilon+\delta)-\frac{c}{4}\varepsilon^3
= \varepsilon - \frac{c}{8}\varepsilon^3 < \varepsilon.
\]
Since $T$ is increasing, $|T(y)-T(x)|=T(y)-T(x)<\varepsilon$. Therefore $T$ is Meir--Keeler. Finally, $(T^n)'(\tfrac12)=1$ for all $n$, so $\Lip(T^n)\ge 1$ and $T$ is not logically contractive.

\smallskip
\textbf{(3) ANE $\not\Rightarrow$ LC.} For $T=\mathrm{Id}_X$, set $k_n:=1+\tfrac1n$. Then $d(T^n x,T^n y)=d(x,y)\le k_n d(x,y)$ with $k_n\to 1$, so $T$ is asymptotically nonexpansive. But $\Lip(T^n)=1$ for all $n$, hence $T$ is not logically contractive.
\end{proof}

\begin{example}[Illustration of Corollary~\ref{cor:rate-nogap}]\label{ex:T-square}
For the map $T$ in Proposition~\ref{prop:incomparability}(1), we have $T^2\equiv 0$. Thus for any $x$ and any $n\ge 2$,
\[
d(T^n x,\,0)\;=\;0\;\le\; \lambda^{\lfloor n/2\rfloor} |x| \qquad\text{for any fixed }\lambda\in(0,1),
\]
in agreement with Corollary~\ref{cor:rate-nogap} using the canonical schedule with $n_1=2$.
\end{example}

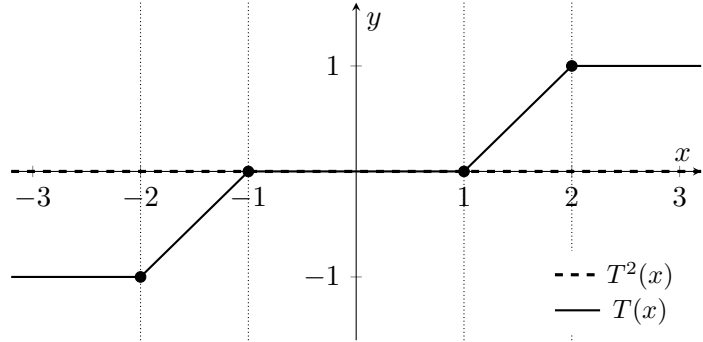
\begin{figure}[t]
\centering
\begin{tikzpicture}
\begin{axis}[
    axis lines=middle,
    xlabel={$x$},
    ylabel={$y$},
    xmin=-3.2, xmax=3.2,
    ymin=-1.6, ymax=1.6,
    xtick={-3,-2,-1,0,1,2,3},
    ytick={-1,0,1},
    width=0.85\linewidth,
    height=0.48\linewidth,
    legend style={draw=none, font=\small, at={(0.98,0.02)}, anchor=south east},
    clip=false
]
\addplot[dashed, very thick, domain=-3.2:3.2, samples=2] {0};
\addlegendentry{$T^{2}(x)$}

\addplot[thick, domain=-3.2:-2, samples=2] {-1};
\addplot[thick, domain=-2:-1, samples=2] {x + 1};
\addplot[thick, domain=-1:1, samples=2] {0};
\addplot[thick, domain=1:2, samples=2] {x - 1};
\addplot[thick, domain=2:3.2, samples=2] {1};
\addlegendentry{$T(x)$}

\addplot[only marks, mark=*, mark size=2pt] coordinates {(-2,-1) (-1,0) (1,0) (2,1)};

\addplot[densely dotted] coordinates {(-2,-1.6) (-2,1.6)};
\addplot[densely dotted] coordinates {(-1,-1.6) (-1,1.6)};
\addplot[densely dotted] coordinates {(1,-1.6) (1,1.6)};
\addplot[densely dotted] coordinates {(2,-1.6) (2,1.6)};
\end{axis}
\end{tikzpicture}
\caption{Piecewise map $T$ used in Proposition~\ref{prop:incomparability}\,(1): $T(x)=0$ for $|x|\le 1$; $T(x)=x-\sgn(x)$ for $1<|x|<2$; $T(x)=\sgn(x)$ for $|x|\ge 2$. The iterate $T^2$ is identically $0$ (dashed).}
\label{fig:example}
\end{figure}

\section{Variable-factor generalization and convergence criterion}
\begin{definition}[Variable-factor logically contractive (VLC)]\label{def:vlc}
A nonexpansive map $T:X\to X$ is \emph{variable-factor logically contractive} if there exist an increasing sequence $n_1<n_2<\cdots$ and factors $\lambda_k\in(0,1]$ such that for all $k\ge 1$ and all $x,y\in X$,
\[
d\!\big(T^{\,n_k}x,\;T^{\,n_k}y\big)\;\le\;\Big(\prod_{i=1}^{k}\lambda_i\Big)\, d(x,y).
\]
Equivalently, with $\Lambda_k:=\prod_{i=1}^{k}\lambda_i$,
\[
d\!\big(T^{\,n_k}x,\;T^{\,n_k}y\big)\;\le\;\lambda_k\, d\!\big(T^{\,n_{k-1}}x,\;T^{\,n_{k-1}}y\big)\quad(k\ge 2),\qquad
d\!\big(T^{\,n_1}x,\;T^{\,n_1}y\big)\;\le\;\lambda_1 d(x,y).
\]
\end{definition}

\noindent\textbf{Fixed point for VLC.}
Let $k_\ast:=\min\{k\ge 1:\Lambda_k<1\}$ (the first strict event). Then $T^{\,n_{k_\ast}}$ is a Banach contraction with a unique fixed point $z$. The same argument as in Theorem~\ref{thm:fixedpoint}(1) yields $Tz=z$ since
\[
T^{\,n_{k_\ast}}(Tz)=T(T^{\,n_{k_\ast}}z)=Tz.
\]

\begin{theorem}\label{thm:vlc}
Let $T$ be VLC with factors $(\lambda_k)$, events $(n_k)$, and fixed point $z$ defined via $k_\ast$. Then:
\begin{enumerate}
\item For all $x\in X$ and $k\ge 1$,
\[
d\!\big(T^{\,n_k}x,\;z\big)\;\le\;\Lambda_k\, d(x,z).
\]
\item If $\Lambda_k\to 0$ (equivalently $\prod_{k=1}^\infty \lambda_k=0$, or $\sum_{k=1}^\infty -\ln\lambda_k=\infty$), then $T^n x\to z$ for every $x\in X$.
\item If $\sup_k (n_{k+1}-n_k)\le M<\infty$, then for all $n\ge n_1$,
\[
d\!\big(T^{\,n}x,\;z\big)\;\le\;\Lambda_{\,1+\left\lfloor \frac{\,n-n_1\,}{M}\right\rfloor}\, d(x,z).
\]
\end{enumerate}
\end{theorem}

\begin{proof}
(1) As in Theorem~\ref{thm:fixedpoint}(3), using $\Lip(T^{\,n_k})\le \Lambda_k$ and $T^{\,n_k}z=z$.
(2) If $n_k\le n< n_{k+1}$, then by nonexpansiveness $d(T^n x,z)\le d(T^{\,n_k}x,z)\le \Lambda_k d(x,z)\to 0$.
(3) As in Corollary~\ref{cor:rate}, take $m=\lfloor (n-n_1)/M\rfloor$ so that $n_{m+1}\le n$, and use (1).
\end{proof}

\begin{lemma}[Product--sum equivalence]\label{lem:prod-sum}
Let $(\lambda_k)_{k\ge1}\subset (0,1]$ and define $\Lambda_k:=\prod_{i=1}^k\lambda_i$. Then $\Lambda_k\to 0$ if and only if $\sum_{k=1}^\infty -\ln\lambda_k=\infty$.
\end{lemma}

\begin{proof}
Set $S_n:=-\sum_{k=1}^n \ln\lambda_k\in[0,\infty]$. Then $\Lambda_n=e^{-S_n}$. Since each $-\ln\lambda_k\ge 0$, the sequence $(S_n)$ is nondecreasing. Hence $\Lambda_n\to 0$ iff $S_n\to\infty$, i.e.\ iff $\sum_{k=1}^\infty -\ln\lambda_k=\infty$.
\end{proof}

\begin{remark}[Borderline and oscillatory factors]\label{rem:borderline}
The condition $\Lambda_k\to 0$ is sufficient for convergence but not necessary in general. If $\Lambda_k$ does not tend to $0$ (e.g.\ it oscillates with $\inf_k \Lambda_k>0$), the eventwise bound
$d(T^{\,n_k}x,z)\le \Lambda_k d(x,z)$ alone cannot force $d(T^{\,n_k}x,z)\to 0$. For instance, taking $\lambda_k=1-\frac{1}{k^2}$ yields $\sum -\ln\lambda_k<\infty$ and $\Lambda_k\downarrow \Lambda_\ast>0$; the theory here gives no convergence guarantee (though a specific map may still converge for other reasons). Conversely, sparse but sufficiently strong events with $\sum -\ln\lambda_k=\infty$ ensure $\Lambda_k\to 0$ and hence convergence.
\end{remark}

\section{Conclusion and outlook}
We introduced logically contractive mappings as a versatile generalization of strict contractions, and established that they preserve the essential benefits of Banach's theorem: existence and uniqueness of a fixed point and global convergence of iterates with geometric-type rates. This framework encompasses cases (like Example~\ref{ex:T-square} and Example~\ref{ex:linfty}) where a map is not a contraction yet converges in finitely many steps, as well as more subtle cases captured by an infinite sequence of intermittent contractions.

Future directions include: extending analysis to broader settings (e.g.\ uniformly convex Banach spaces and stochastic schedules), enriching the comparative landscape with further explicit examples, and optimizing per-iterate rates from finer statistics of event distributions beyond worst-case gaps.

\end{document}